\documentclass[10pt,a4paper]{article}

\usepackage{amsmath,amsthm}
\usepackage[numbers]{natbib}
\usepackage{microtype}
\usepackage[charter]{mathdesign}
\usepackage{color}
\usepackage[final]{showkeys}
\usepackage{yfonts}

\numberwithin{equation}{section}
\usepackage[hmargin=3cm,vmargin={3.5cm,4cm}]{geometry}
\allowdisplaybreaks

\newtheorem{theorem}{Theorem}[section]
\newtheorem{thm}{Theorem}[section]
\newtheorem{remark}{Remark}[section]
\newtheorem{rem}{Remark}[section]
\newtheorem{lem}{Lemma}[section]
\newtheorem{definition}{Definition}[section]

\DeclareMathAlphabet{\mathpzc}{OT1}{pzc}{m}{it}

\title{Fractional diffusions with time-varying coefficients}

\author{$\text{Roberto Garra}_1$, $\text{Enzo Orsingher}_2$, $\text{Federico Polito}_3$ \\
    \footnotesize (1) -- Dipartimento di Scienze Statistiche, ``Sapienza'' Universit\`{a} di Roma\\
        \footnotesize P.~le A.\ Moro 5, 00185 Roma, Italy\\
    \footnotesize Email address: roberto.garra@sbai.uniroma1.it\\
    \footnotesize (2) -- Dipartimento di Scienze Statistiche, ``Sapienza'' Universit\`{a} di Roma\\
    \footnotesize P.~le A.\ Moro 5, 00185 Roma, Italy\\
    \footnotesize Email address: enzo.orsingher@uniroma1.it\\
    \footnotesize (3) -- Dipartimento di Matematica ``G.~Peano'', Universit\`{a} degli Studi di Torino\\
    \footnotesize Via Carlo Alberto 10, 10123 Torino, Italy\\
    \footnotesize Email address: federico.polito@unito.it\\
    }

\begin{document}

    \maketitle

    \begin{abstract}

	    This paper is concerned with the fractionalized diffusion equations governing the law of the fractional 
	    Brownian motion $B_H(t)$. We obtain solutions of these equations which are probability laws extending that of $B_H(t)$.
	    Our analysis is based on McBride fractional operators generalizing the hyper-Bessel operators $L$ and converting
	    their fractional power $L^{\alpha}$ into Erd\'elyi--Kober fractional integrals. We study also probabilistic 
	    properties of the r.v.'s whose distributions satisfy space-time fractional equations involving Caputo and Riesz 
	    fractional derivatives. Some results emerging from the analysis of fractional equations with time-varying coefficients 
	    have the form of distributions of time-changed r.v.'s.
	

    \end{abstract}
    
    \section{Introduction}

	    In this paper we consider fractional diffusion equations with time-varying coefficients. This investigation 
	    is inspired by the fractional extension of the diffusion equation governing the law of the fractional Brownian motion 
	    $B_{H}(t)$, $t\geq 0$, $0<H<1$. Fractional Brownian motion is a zero-mean Gaussian process with covariance function
	    \begin{equation}
    		\mathbb{C}\text{ov}\left(B_{H}(s), B_H(t)\right)= \frac{1}{2}\bigg\{|t|^{2H}+|s|^{2H}-|t-s|^{2H}\bigg\}, \quad t,s>0.
	    \end{equation}
	    
	    We first deal with equations of the form 
		\begin{equation}
			\label{iuno}
			\left(t^{1-2H}\frac{\partial}{\partial t}\right)^{\alpha}u_{\alpha}(x,t)=H^{\alpha}
			\frac{\partial^2}{\partial x^2}u_{\alpha}(x,t),\qquad \alpha\in(0,1), \ H\in (0,1), \ x \in \mathbb{R},
		\end{equation}
		where the operator appearing in \eqref{iuno} is a special case of the fractional power $L^{\alpha}$ of the 
		hyper-Bessel-type operator
		\begin{equation}
		L= t^{a_1}\frac{\mathrm{d}}{\mathrm{d}t}t^{a_2}\frac{\mathrm{d}}{\mathrm{d}t}\dots\frac{\mathrm{d}}{\mathrm{d}t}t^{a_{n+1}},\qquad t>0,
		\end{equation}
		where $a_j$, $j= 1, \dots, n+1$, are real numbers and $n\in \mathbb{N}$.

		For $\alpha = 1$, \eqref{iuno} coincides with the equation governing the one-dimensional probability law of 
		$B_H(t)$, while for $H= 1/2$, \eqref{iuno} is the classical time-fractional diffusion equation.
		There is a considerably large literature about fractional diffusion equations with constant coefficients, dating back to the second half
		of the Eighties (\cite{Fujita,Wyss}, for a detailed review see \cite{mainardi}). Amongst the more recent papers we refer, for example, to
		\cite{Turner, Luc, meer}.
		In some specific cases, the relationship 
		between the iterated Brownian motion and fractional diffusion equations was established and fully investigated \cite{Allouba, Ors}.
		
		The theory of powers of Bessel-type operators was developed 
		by Dimovski \cite{Dim66}, McBride \cite{mc2,mc1,mca}, McBride and Lamb \cite{lam}. The fractional
		powers of Bessel operators were recently considered in \cite{RG} in the analysis of
		random motions at finite velocity.
	
		The theory developed by McBride permits us to express the operator $L^{\alpha}$ in terms of products
		of Erd\'ely--Kober fractional integrals $I_m^{\alpha}$, defined as 
		\begin{equation}
			\left(I_m^{\alpha}f\right)(t)= \frac{m}{\Gamma(\alpha)}\int_0^t\left(t^m-u^m\right)^{\alpha-1}u^{m-1}f(u)
			\, \mathrm{d}u, \qquad \alpha >0, t>0, m>0.
		\end{equation}
	
		As far as we are aware, the generalization of time-fractional equations with time-varying coefficients by means of the McBride 
		theory have not been studied so far, neither 
		from an analytical nor from a probabilistic point of view. 
	
		The fundamental solution of \eqref{iuno} can be written as 
		\begin{equation}
			\label{idue}
			u_{\alpha}(x,t)=\frac{1}{2^{1-\alpha/2}t^{H\alpha}}
			W_{-\alpha/2, 1-\alpha/2}\left(-\frac{2^{\alpha/2}|x|}{t^{H\alpha}}\right), \qquad x \in \mathbb{R},
		\end{equation}
		where 
		\begin{equation}
			W_{\gamma, \beta}(x)= \sum_{k=0}^{\infty}\frac{x^k}{k!\Gamma(\gamma k +\beta)}, \qquad \gamma >-1, \beta>0, x\in \mathbb{R},
		\end{equation}
		is the Wright function.
	
		Formula \eqref{idue} shows that the solution $u_{\alpha}(x,t)$ of \eqref{iuno} coincides with the fundamental solution 
		of the time-fractional diffusion equation with the time-scale change $t\rightarrow t^{2H}$.
		We prove that $u_{\alpha}(x,t)$ is the law of the r.v.\ $X_{\alpha, H}(t)$ which is connected to the 
		fractional Brownian motion by means of the relation
		\begin{equation}
			X_{\alpha, H}(t)\overset{d}{=}X_{2\alpha,H}\left(|B_H(t)|^{\frac{1}{2H}}\right), \quad 0<\alpha<\frac{1}{2}.
		\end{equation} 
		
		We note that the r.v.\ $X_{\alpha, H}(t)$ has variance
		\begin{equation}
	 	\label{ivari}
	 		\mathbb{V}\text{ar} \, X_{2\alpha, H}(t)=
	 		\frac{t^{2H\alpha}}{2^{\alpha-1}\Gamma(\alpha+1)}
	 	\end{equation}
		and thus the presence of the fractional derivative has a decreasing effect on the variance.
	 
		In Section \ref{mala} we establish a relationship between solutions of fractional equations
		\begin{equation}
			\left(t^{1-2H}\frac{\partial}{\partial t}\right)^{\alpha}u = H^{\alpha}\frac{\partial^2 u}{\partial x^2}
		\end{equation} 
		and solutions of higher order diffusion equations
		\begin{equation}
			t^{1-2H}\frac{\partial u}{\partial t}= (-1)^k H \frac{\partial^k u}{\partial x^k}, \qquad k>2.
		\end{equation}
		This type of relationship frequently appeared in the recent literature in the case of fractional and 
		higher-order diffusion equations with constant coefficients (see e.g. \cite{mirko1}).	 
		
		In the second part of this paper, we consider 
		fractional equations of the form 
		\begin{equation}
			\label{imultifr}
	 		\begin{cases}
	 			{}^C D^\nu_{0^+} u \left( x,t \right) =
	 			H t^{2H-1} \frac{\partial^2}{\partial x^2} u \left( x,t \right),
	 			\qquad \nu \in (0,1), \: x \in \mathbb{R}, \: t >0,\\
	 			u \left( x,0 \right) = \delta \left( x \right),
	 		\end{cases}
	 	\end{equation}
		where ${}^C D^\nu_{0^+}$ is the $\nu$-th order Caputo time-fractional derivative. This problem was partially considered in \cite{fepr}.
		The idea is to compare the two approaches \eqref{iuno} and \eqref{imultifr}.
	
		The fundamental solution of \eqref{imultifr} reads
		\begin{align}
			\label{isolutiona}
				u(x,t)
				= \frac{1}{2 \pi} \int_{-\infty}^{+\infty} e^{-i \beta x}
				E_{\nu, 1+\frac{2H-1}{\nu}, \frac{2H-1}{\nu}} \left(
				- H \beta^2 t^{\nu+2H-1} \right) \textrm d\beta. 
		\end{align}
		where the function
		\begin{align}
			\label{imlgen}
			E_{\alpha, m, l} (z) &= 1+ \sum_{k=1}^\infty z^k \prod_{j=0}^{k-1}
			\frac{ \Gamma \left( \alpha \left( jm+l \right)+1 \right)}{\Gamma \left( \alpha
			\left( jm+l+1 \right) +1 \right)}, \qquad z \in \mathbb{R},\\
			\nonumber &\text{$\alpha, m, l\in \mathbb{R}$ such that $\alpha, m >0$ and $\alpha(jm+l)\neq \mathbb{Z}^-$},
		\end{align}
	    is the generalized three-index Mittag--Leffler function first introduced by Kilbas and Saigo in \cite{kilsa1,kilsa2,kilsa3};
		see also the recent monograph \cite{Rogosin}.
		
		Our research is related to the recent papers \cite{Bologna1, Bologna2},
		where fractional generalizations of the heat equation for the fractional Brownian motion were inspired
		by some physical problems. In particular, in these papers the authors study 
		\begin{equation}
			\label{grigo}
 			{}^C D^\nu_{0^+} u \left( x,t \right) =
 			D_H t^{2H-1} \frac{\partial^\beta}{\partial |x|^{\beta}} u \left( x,t \right),
 			\qquad \nu \in (0,1), \: x \in \mathbb{R}, \: t >0,
	 	\end{equation}	
		in order to take into account the joint action of trapping and correlated fluctuations for anomalous diffusions in heterogeneous media.
		In this framework we recover some of the analytical results discussed in \cite{Bologna2}. For the special case $H= 1/4$ and $\nu= 3/4$, 
		we are able to give an explicit and interesting form for the fundamental solution of \eqref{grigo}
		as the law of a time-changed Brownian motion.

	\section{Preliminaries}

		In this section we recall the main definitions and properties of fractional derivatives and fractional powers of
		hyper-Bessel-type differential operators for the convenience of the reader. There is a large literature
		about fractional powers of differential operators. We refer, for example, to the monographs \cite{Martinez} and \cite{mc1}.
		We start our short review recalling the definitions of Riemann--Liouville fractional integrals and derivatives.
		\begin{definition}[Riemann--Liouville integral]
		    Let $f \in L^1_{\text{loc}}[a,b]$, where $-\infty \le a < t < b \le \infty$,
            be a locally integrable real-valued function.
			The Riemann--Liouville integral is defined as
			\begin{align}
				\label{rlint}
			    I^{\alpha}_{a^+} f(t) & =\frac{1}{\Gamma(\alpha)}\int_{a}^t\frac{f(u)}{%
			    (t-u)^{1-\alpha}}\mathrm du
			    = (f \ast K_\alpha )(t), \qquad \alpha > 0,
			\end{align}
			where $K_\alpha (t) = t^{\alpha-1}/\Gamma(\alpha)$.
		\end{definition}
	
		\begin{definition}[Riemann--Liouville derivative]
		    Let $f \in L^1[a,b]$, $-\infty \le a < t < b \le \infty$, and $f \ast K_{m-\alpha} \in W^{m,1}[a,b]$,
                $m = \lceil \alpha \rceil$,
		    $\alpha>0$,
		    where $W^{m,1}[a,b]$ is the Sobolev space defined as
		    \begin{align}
		        W^{m,1}[a,b] = \left\{ f \in L^1[a,b] \colon \frac{\mathrm d^m}{\mathrm dt^m} f \in L^1[a,b] \right\}.
		    \end{align}
		    The Riemann--Liouville derivative of order $\alpha >0$ is defined as
		    \begin{align}
		        \label{rlder}
		        D^\alpha_{a^+}f (t) = \frac{\mathrm d^m}{\mathrm dt^m}I_{a^+}^{m-\alpha}f(t) = \frac{1}{%
		        \Gamma(m-\alpha)} \frac{\mathrm d^m}{\mathrm dt^m} \int_{a}^t (t-s)^{m-1-\alpha}f(s) \mathrm ds.
		    \end{align}
		\end{definition}
		
		It is a simple matter to show that the Riemann--Liouville 
		fractional derivative is rigorously the fractional power of the operator $D = \mathrm d /\mathrm d t$
		in the classical theory of fractional powers of operators (see e.g.\ \cite{mirko} and the 
		references therein). We now recall that the so called Caputo derivative, widely used in applications, is in fact
		a regularization of the Riemann--Liouville derivative.
		For $n\in \mathbb{N}$, we denote by $AC^{n}\left[a,b\right]$ the
		space of real-valued functions $f\left( t\right) $, $t>0$, which have
		continuous derivatives up to order $n-1$ on $\left[ a,b\right] $
		such that $f^{\left(n-1\right) }\left(t\right)$ belongs to the space of absolutely continuous functions
		$AC\left[a,b\right]:$
		\begin{equation}
			AC^{n}\left[a,b\right] =\left\{ f:\left[a,b\right] \rightarrow
			\mathbb{R}\colon\frac{\mathrm d^{n-1}}{\mathrm dx^{n-1}}f \left( x\right) \in AC\left[
			a,b\right] \right\} .
		\end{equation}
			
		\begin{definition}[Caputo derivative]
		    Let $\alpha>0$, $m = \lceil \alpha \rceil$, and $f \in AC^m[a,b]$.
		    The Caputo derivative of order $\alpha>0$ is defined as
		    \begin{equation}
		        \label{Capu}
		        {}^C D^{\alpha}_{a^+}f(t)= I_{a^+}^{m-\alpha}\frac{\mathrm d^m}{\mathrm dt^m}f(t)= \frac{1%
		        }{\Gamma(m-\alpha)}\int_a^{t}(t-s)^{m-1-\alpha}\frac{\mathrm d^m}{\mathrm ds^m}f(s) \, \mathrm ds.
		    \end{equation}
		\end{definition}

		Note that in the space of functions belonging to $AC^m[a,b]$ the 
		relation between Riemann--Liouville and Caputo derivatives is given by the following

		\begin{theorem}
		    \label{gianduia}
		    For $f \in AC^m[a,b]$, $m=\lceil \alpha \rceil$, $\alpha\in\mathbb{R}^+\backslash\mathbb{N}$,
		    the Caputo fractional derivative of order $\alpha$ of $f$
		    exists almost everywhere and it can be written as
		    \begin{align}
		        \label{perepe}
		      {}^C  D_{a^+}^\alpha f(t) =  D_{a^+}^\alpha\left( f(t) - \sum_{k=0}^{m-1}\frac{(t-a)^{k}}{k!} f^{(k)}(a^+)\right)
		      = D_{a^+}^\alpha f(t)-\sum_{k=0}^{m-1}\frac{(t-a)^{k-\alpha}}{\Gamma(k+1-\alpha)} f^{(k)}(a^+).
		    \end{align}
		\end{theorem}
		For the proof of Theorem 2.1 consult \cite{diet}, page 50, Theorem 3.1.\\
		We observe that the difference between the two definitions is given by 
		the Riemann--Liouville derivative of the Taylor polynomial of order $m-1$ centered in the lower extremum of integration
		$t = a$. Hereafter without loss of generality we will consider
		$a= 0$.
		
		In view of this well-known relation between the two different fractional derivatives, we now recall the 
		definition of fractional hyper-Bessel-type differential operators, studied in a series of works
		by McBride \citep{mc2,mc1,mca} and McBride and Lamb \citep{lam}.
        In \cite{mca} McBride considered the generalized hyper-Bessel operator
        \begin{equation}
            \label{L}
            L=t^{a_1}\frac{\mathrm d}{\mathrm{d}t}t^{a_2}\dots t^{a_n}\frac{\mathrm d}{\mathrm{d}t}t^{a_{n+1}},\qquad t>0,
        \end{equation}
        where $n$ is an integer number and $a_1,\dots, a_{n+1}$ are real numbers.
        
        The operator $L$ defined in \eqref{L} acts on the functional space
        \begin{equation}
            F_{p,\mu}=\{f \colon t^{-\mu}f(t)\in F_p\},
        \end{equation}
        where
        \begin{equation}
            F_p=\left\{f\in C^{\infty} \colon t^k \frac{\mathrm d^k f}{\mathrm dt^k} \in L^p(0, \infty), \: k=0,1,\dots\right\},
        \end{equation}
        for $1 \leq p < \infty$ and for any real number $\mu$ (see for details \cite{mc2, mc1}).
        The operator $L$ was first introduced and studied, as far as we know, by Dimovski \cite{Dim66}.
         
        Let us introduce the following constants related to the general operator $L$.
        \begin{align*}
			a=\sum_{k=1}^{n+1}a_k, \qquad m= |a-n|,
			\qquad b_k= \frac{1}{m}\left(\sum_{i=k+1}^{n+1}a_i+k-n\right), \quad k=1, \dots, n.
        \end{align*}
        The definition of the fractional hyper-Bessel-type operator is given by
  	    \begin{definition}
        Let $m= n-a>0$, $f\in F_{p,\mu}$ and
        \begin{align*}
            b_k\in A_{p,\mu,m}:=\left\{\eta \in \mathbb{R}
            \colon m\eta+\mu+m\neq \frac{1}{p}-ml, \: l=0, 1, 2,\dots\right\}, \qquad k=1,\dots, n.
        \end{align*}
        Then
        \begin{equation}
         	\label{pot}
            L^{\alpha}f=m^{n\alpha}t^{-m\alpha}\prod_{k=1}^{n}I_{m}^{b_k,-\alpha}f,
        \end{equation}
        where, for $\alpha >0$ and $m\eta+\mu+m > \frac{1}{p}$
        \begin{equation}
            \label{mc1-2}
            I_m^{\eta, \alpha}f=
            \frac{t^{-m\eta-m\alpha}}{\Gamma(\alpha)}\int_0^t(t^m-u^m)^{\alpha-1}u^{m\eta}f(u)\, \mathrm{d}(u^m),
        \end{equation}
        and for $\alpha\leq 0$
        \begin{equation}
            \label{mc2}
              	I_m^{\eta, \alpha}f=(\eta+\alpha+1)I_m^{\eta, \alpha+1}f+\frac{1}{m} I_m^{\eta, \alpha+1}
              	\left(t\frac{\mathrm{d}}{\mathrm{d}t}f\right).
  	        \end{equation}
  	    \end{definition}
        For a full discussion about this approach to fractional operators we refer to \cite{mca}, page 525.
        According to the analysis developed in this paper, the domain of the operators 
        $L^{\alpha}$ is the space $F_{p,\mu}$.
        
        A key-role in the next sections is played by the regularized Caputo-like counterpart of the operator \eqref{pot}.
        In analogy with Theorem \ref{gianduia} we introduce the following 
        \begin{definition}
        	Let $\alpha$ be a positive real number, $m= n-a>0$, $f\in F_{p,\mu}$ is such that
            \begin{equation*}
            	L^\alpha\left( f(t) - \displaystyle\sum_{k=0}^{b-1}\frac{t^{k}}{k!} f^{(k)}(0^+)\right)
            \end{equation*}
            exists. Then we define ${}^C L^\alpha$ by
            \begin{equation}
	            \label{pot1}
	            {}^C  L^\alpha f(t) =  L^\alpha\left( f(t) - \sum_{k=0}^{b-1}\frac{t^{k}}{k!} f^{(k)}(0^+)\right),
            \end{equation}
            where $b = \lceil \alpha \rceil$.
        \end{definition}
        
        The relevance of this definition for the applications is due to the fact that for physical reasons we are interested in solving
        fractional Cauchy problems involving initial conditions on the functions (and their integer order derivatives),
        while in the case of Riemann--Liouville operators we must consider initial conditions involving fractional integrals
        and derivatives.

    \section{Fractional diffusions with time-dependent coefficients}

		Let us consider the fractional diffusion equation with time-dependent coefficients
		\begin{equation}
			\label{uno}
			\left(t^{1-2H}\frac{\partial}{\partial t}\right)^{\alpha}u_{\alpha}(x,t)=H^{\alpha}\frac{\partial^2}{\partial x^2}u_{\alpha}(x,t),
		\end{equation}
		where $H\in (0,1)$ is the Hurst parameter, $\alpha \in (0,1)$ and, according to Definition 2.4, the fractional operator is defined
		(for $n=1$, $a_1 = 1-2H$, $a_2 = 0$ and thus $b_1 = 0$, $m= 2H$) as
		\begin{equation}
			\label{cia}
			\left(t^{1-2H}\frac{\partial}{\partial t}\right)^{\alpha} u_{\alpha}(x,t)=(2H)^{\alpha}t^{-2H\alpha}
			I^{0,-\alpha}_{2H}u_{\alpha}(x,t),
		\end{equation}
		where $I^{0,-\alpha}_{2H}$ is the Erd\'elyi--Kober fractional integral defined in \eqref{mc1-2}.\\
		According to \cite{mc1}, the fractional power of the operator
		$\left(t^{1-2H}\partial/\partial t\right)$ admits the representation \eqref{cia} in the domain $F_{p,\mu}$,
		where $\mu \in \mathbb{R}$ should be taken such that
		$$b_1 = 0\subseteq A_{p,\mu, 2H}:\{\eta \in \mathbb{R}:
		2H\eta+\mu +2H\neq\frac{1}{p}-2Hl, \ l= 0,1,\dots \}.$$
		This constraint gives the implicit relation between the parameter $\mu$ of the domain of the operator and the parameter
		$H$ of the specific operator that we are considering here.
		
		We can call $\eqref{uno}$ a bifractional equation, depending on the Hurst parameter $H$ (which, in the case $\alpha = 1$, is related to
		the fractional Brownian motion) and the real parameter $\alpha$ (that is the fractional power of the operator).
		Notice that for $\alpha = 1$ we retrieve the heat equation with one time-dependent coefficient,
		\begin{equation}
			t^{1-2H}\frac{\partial u}{\partial t}=H\frac{\partial^2 u}{\partial x^2},
		\end{equation}
		satisfied by the probability law of the fractional Brownian motion.
		
		In view of definition 2.5 the regularized Caputo-like counterpart of the fractional operator \eqref{cia} reads
		\begin{equation}
			\label{ci}
			{}^C\left(t^{1-2H}\frac{\partial}{\partial t}\right)^{\alpha} u_{\alpha}(x,t)=
			\left(t^{1-2H}\frac{\partial}{\partial t}\right)^{\alpha}u_{\alpha}(x,t)
			-(2H)^{\alpha}\frac{t^{-2H\alpha}}{\Gamma(1-\alpha)}u_{\alpha}(x,0),
		\end{equation}
		in complete analogy with the classical theory of fractional derivatives.
		\begin{theorem}
			The solution to the Cauchy problem
			\begin{equation}
				\begin{cases}
					\label{cas}
					{}^C\left(t^{1-2H}\frac{\partial}{\partial t}\right)^{\alpha} u_{\alpha}(x,t)= H^{\alpha}
					\frac{\partial^2}{\partial x^2}u_{\alpha}(x,t),
					\quad \alpha \in (0,1), \: x\in \mathbb{R} \ t > 0,\\
					u_{\alpha}(x,0)=\delta(x),
				\end{cases}
			\end{equation}
			is given by
			\begin{equation}
				\label{due}
				u_{\alpha}(x,t)=\frac{1}{2^{1-\alpha/2}t^{H\alpha}}
				W_{-\alpha/2, 1-\alpha/2}\left(-\frac{2^{\alpha/2}|x|}{t^{H\alpha}}\right).
			\end{equation}
		\end{theorem}
		\begin{proof}
			By using \eqref{ci}
			we obtain the equation
			\begin{equation}
				\label{prova}
				\left(t^{1-2H}\frac{\partial}{\partial t}\right)^{\alpha} u_{\alpha}(x,t)= H^{\alpha}\frac{\partial^2}{\partial x^2}u_{\alpha}(x,t)+
				(2H)^{\alpha} \frac{t^{-2H\alpha}}{\Gamma(1-\alpha)}\delta(x).
			\end{equation}
			The Fourier transform of \eqref{prova} reads
			\begin{equation}
				\label{pr}
				\left(t^{1-2H}\frac{\partial}{\partial t}\right)^{\alpha} \widehat{u}_{\alpha}(\beta,t)
				= -H^{\alpha}\beta^2 \widehat{u}_{\alpha}(\beta,t)+
				(2H)^{\alpha}\frac{t^{-2H\alpha}}{\Gamma(1-\alpha)}.
			\end{equation}
			The solution to \eqref{pr} can be easily determined as
			\begin{equation}
				\label{la}
				\widehat{u}_{\alpha}(\beta,t)= E_{\alpha,1}\left(-\frac{\beta^2 t^{2H\alpha}}{2^{\alpha}}\right),
			\end{equation}
			where $E_{\alpha,1}(\cdot)$ is the Mittag--Leffler function.
			Hence, by using \eqref{cia}, we have that
			\begin{align}
				&\left(t^{1-2H}\frac{\partial}{\partial t}\right)^{\alpha}\widehat{u}_{\alpha}(\beta,t)=(2H)^{\alpha}t^{-2H\alpha}
				I^{0,-\alpha}_{2H}\widehat{u}_{\alpha}(\beta,t)\\
				\nonumber & =(2H)^{\alpha}t^{-2H\alpha}
				I^{0,-\alpha}_{2H} \sum_{k=0}^{\infty}\left(-\frac{\beta^2 t^{2H\alpha}}{2^{\alpha}}\right)^k
				\frac{1}{\Gamma(\alpha k+1)}=(2H)^{\alpha}\sum_{k=0}^{\infty}
				\left(-\frac{\beta^2}{2^{\alpha}}\right)^k
				\frac{t^{2H\alpha k-2H\alpha}}{\Gamma(\alpha k+1-\alpha)}\\
				\nonumber &= -H^{\alpha}\beta^2 E_{\alpha,1}\left(-\frac{\beta^2 t^{2H\alpha}}{2^{\alpha}}\right)+
				(2H)^{\alpha}\frac{t^{-2H\alpha}}{\Gamma(1-\alpha)}=
				-H^{\alpha}\beta^2\widehat{u}_{\alpha}(\beta,t)+(2H)^{\alpha}\frac{t^{-2H\alpha}}{\Gamma(1-\alpha)},
			\end{align}
			where we used the fact that
			\begin{equation}
				I^{0,-\alpha}_{2H}t^{2\alpha Hk}=\frac{\Gamma(\alpha k+1)}{\Gamma(\alpha k+1-\alpha)}t^{2\alpha Hk}.
			\end{equation}
			By inverting the Fourier transform $\widehat{u}(\beta,t)$ we obtain the claimed result.
		\end{proof}
		
		We note that the time-change $t\rightarrow t^{2H}$ reduces equation \eqref{uno} to the classical time-fractional diffusion
		equation 
		\begin{equation}
				{}^C D^\alpha_{0^+}u(x,t)=\frac{1}{2^{\alpha}}\frac{\partial^2}{\partial x^2}u(x,t),
		\end{equation}
		with fundamental solution given by (see e.g.\ \cite{Ors})
		\begin{equation}
			u_{\alpha}(x,t)=\frac{1}{2^{1-\alpha/2}t^{\frac{\alpha}{2}}}
			W_{-\alpha/2, 1-\alpha/2}\left(-\frac{2^{\alpha/2}|x|}{t^{\frac{\alpha}{2}}}\right).
		\end{equation}
		This is a particularly interesting fact. In the first part of the paper in fact we prove formally that suitably time-changed
		fractional diffusion processes can be analyzed by means of the
		McBride theory. The converse is not true. More general equations can be treated by applying the same theory but
		cannot be recovered by means of simple deterministic time-changes.

		\begin{theorem}
			The solution to
			\begin{equation}
				\begin{cases}
					\label{pri}
					{}^C\left(t^{1-2H}\frac{\partial}{\partial t}\right)^{\alpha} u_{\alpha}(x,t)= H^{\alpha}
					\frac{\partial^2}{\partial x^2}u_{\alpha}(x,t),
					\quad \alpha \in (0,1), \: x\in\mathbb{R}, \: t>0,\\
					u_{\alpha}(x,0)=\delta(x),
				\end{cases}
			\end{equation}
			can be represented as
			\begin{equation}
				\label{lenovo}
				u_{\alpha}(x,t)=\frac{H\sqrt{2}}{\sqrt{\pi}t^H}\int_0^{\infty}z^{2H-1}
				e^{-\frac{1}{8}\left(\frac{z}{\sqrt{t}}\right)^{4H}}u_{2\alpha}(x,z)\,\mathrm{d}z,
			\end{equation}
			where $u_{2\alpha}$ is the solution to
			\begin{equation}
				\label{sec}
				\begin{cases}
					{}^C\left(t^{1-2H}\frac{\partial}{\partial t}\right)^{2\alpha} u_{2\alpha}(x,t)= H^{2\alpha}
					\frac{\partial^2}{\partial x^2}u_{2\alpha}(x,t),
					\quad \alpha \in \left(0,\frac{1}{2}\right],\\
					u_{2\alpha}\left(x,0\right)=\delta(x),
				\end{cases}
			\end{equation}
			and
			\begin{equation}
				\label{ter}
				\begin{cases}
					{}^C\left(t^{1-2H}\frac{\partial}{\partial t}\right)^{2\alpha} u_{2\alpha}(x,t)= H^{2\alpha}
					\frac{\partial^2}{\partial x^2}u_{2\alpha}(x,t),
					\quad \alpha \in (\frac{1}{2},1),\\
					u_{2\alpha}(x,0)=\delta(x),\\
					\frac{\partial}{\partial t}u_{2\alpha}(x,t)|_{t=0}= 0.
				\end{cases}
			\end{equation}
		\end{theorem}
		\begin{proof}
			The proof follows the main steps of Theorem 2.1 in \cite{Ors}. In more detail,
			by applying the duplication formula of the Gamma function we have
			\begin{equation}
				\label{dup}
				\Gamma\left(-\frac{\alpha k}{2}+1-\frac{\alpha}{2}\right)= \sqrt{\pi}2^{\alpha(k+1)}
				\frac{\Gamma(1-\alpha(k+1))}{\Gamma\left(\frac{1}{2}-\frac{\alpha(k+1)}{2}\right)}.
			\end{equation}
			By substituting \eqref{dup} in \eqref{due} we have
			\begin{align}
				u_{\alpha}(x,t)&=\frac{1}{2^{1-\alpha/2}t^{H\alpha}}
				\sum_{k=0}^{\infty}\frac{\left(-\frac{|x|}{t^{H\alpha}}\right)^k
				\Gamma(\frac{1-\alpha(k+1)}{2})}{k!\sqrt{\pi}2^{\alpha(k+1)-\frac{\alpha k}{2}}\Gamma(1-\alpha(k+1))}\\
				\nonumber &=\frac{1}{2^{1+\alpha/2}\sqrt{\pi}t^{H\alpha}}
				\sum_{k=0}^{\infty}\frac{\left(-\frac{|x|}{t^{H\alpha}}\right)^k
				\int_0^{\infty}e^{-w}w^{-\frac{\alpha}{2}(k+1)-\frac{1}{2}}\mathrm{d}w}{k!2^{\frac{\alpha k}{2}}\Gamma(1-\alpha(k+1))}\\
				\nonumber &=\frac{1}{2^{1+\alpha/2}\sqrt{\pi}t^{H\alpha}}\int_0^{\infty}e^{-w}w^{-\alpha/2-1/2}
				\sum_{k=0}^{\infty}\frac{(-|x|/(wt^{2H})^{\alpha/2})^k
				}{k!2^{\alpha k/2}\Gamma(1-\alpha(k+1))}\mathrm{d}w\\
				\nonumber &=\frac{2^{1-\alpha}}{2^{1+\alpha/2}\sqrt{\pi}t^{H\alpha}}\int_0^{\infty}e^{-w}w^{-\alpha/2-1/2}
				t^{H\alpha}(2^3 w)^{\alpha/2}u_{2\alpha}(x,\sqrt{t}(2^3 w)^{\frac{1}{4H}})\mathrm{d}w\\
				\nonumber &=\frac{1}{\sqrt{\pi}}\int_0^{\infty}e^{-w}w^{-1/2}
				u_{2\alpha}(x,\sqrt{t}(2^3w)^{\frac{1}{4H}})\mathrm{d}w\\
				\nonumber & =[z= \sqrt{t}(2^3w)^{\frac{1}{4H}}]\\
				\nonumber & = \frac{H\sqrt{2}}{\sqrt{\pi}t^{H}}\int_0^{\infty}z^{2H-1}
				e^{-\frac{1}{8}\left(\frac{z}{\sqrt{t}}\right)^{4H}}u_{2\alpha}(x,z)\mathrm{d}z.
			\end{align}
		\end{proof}

		\begin{remark}
			In the special case $\alpha= \frac{1}{2}$, the solution reads
			\begin{equation}
				u_{1/2}(x,t)=\frac{H\sqrt{2}}{\sqrt{\pi}t^{H}}\int_0^{\infty}z^{2H-1}
				e^{-\frac{1}{8}\left(\frac{z}{\sqrt{t}}\right)^{4H}}u_1(x,z) \mathrm dz,
			\end{equation}
			where $u_1(x,t)$ is the fundamental solution of the heat equation
			\begin{equation}
				\left(t^{1-2H}\frac{\partial}{\partial t}\right) u_{1}(x,t)=
				H\frac{\partial^2}{\partial x^2}u_{1}(x,t).
			\end{equation}
			Therefore $u_{1/2}(x,t)$ coincides with the probability density of
			the r.v. $B_H^1\left(|B_{H}^2(t)|^{\frac{1}{2H}}\right)$,
			where $B_H^1$ and $B_H^2$ are independent fractional Brownian
			motions. We note that $\mathbb{V}\text{ar} \,B_H^1(t)= t^{2H}$ and $\mathbb{V}\text{ar} \,B_{H}^2(t)= 4t^{2H}$.
			
			If $H=\frac{1}{2}$, we recover the result discussed in \cite{Ors}, i.e. $u_{1/2}(x,t)$
			coincides with the probability density of the iterated Brownian motion $B^1\left(|B^2(t)|\right)$ ($\mathbb{V}\text{ar} \,B^1(t)= t$ 
			and $\mathbb{V}\text{ar} \,B^2(t)= 4t$) .
        \end{remark}

		\begin{remark}
			We observe that, in the more general case, for $\alpha \in (0,1)$, the solution $u_{\alpha}(x,t)$ coincides
			with the probability density of the r.v. $X_{2\alpha, H}\left(|B_{H}(t)|^{\frac{1}{2H}}\right)$,
			where $X_{2\alpha, H}(t)$, $t\ge 0$, is
			a r.v., independent from $B_H(t)$ with density law equal to
			$u_{2\alpha}(x,t)$.
		\end{remark}
		\begin{remark}
			It is simple to find the moments of the process $X_{\alpha, H}(t)$,
			$t\ge 0$, by means of its characteristic function \eqref{la}. In
			particular, its variance is given by
			\begin{equation}
				\label{vari}
				\mathbb{V}\text{ar} \, X_{\alpha, H}(t)=
				\frac{t^{2H\alpha}}{2^{\alpha-1}\Gamma(\alpha+1)}.
			\end{equation}
			We observe that, for $\alpha =1$ in \eqref{vari}, we recover the
			variance of the fractional Brownian motion; for $H=1/2$ we have
			the variance of the time-fractional diffusion process.
			
			More generally, from \eqref{la}, we infer that 
			\begin{equation}
				\mathbb{E}X^{2m}_{\alpha, H}(t)= \frac{\Gamma(2m+1)}{2^{\alpha m}\Gamma(\alpha m+1)}t^{2H\alpha m}, \quad m \in \mathbb{N}.
			\end{equation}
		\end{remark}
	
		\begin{remark}
			We now consider the more general time-fractional diffusion
			equation with time-varying coefficients
			\begin{equation}
				\label{gen}
				{}^C\left(t^{a_1}\frac{\partial}{\partial
				t}\right)^{\alpha}u_{\alpha}(x,t)= \left(\frac{1-a_1}{2} \right)^\alpha \frac{\partial^2}{\partial
				x^2}u_{\alpha}(x,t), \qquad a_1<1.
			\end{equation}

			In view of the McBride theory we have that
			in this case
			\begin{equation}
				\label{ge}
				\left(t^{a_1}\frac{\partial}{\partial
				t}\right)^{\alpha}f(t)=
				(1-a_1)^{\alpha}t^{-\alpha(1-a_1)}I^{0,-\alpha}_{1-a_1}f(t).
			\end{equation}
			The solution of the Cauchy problem
			\begin{equation}
				\begin{cases}
					\label{casg}
					{}^C\left(t^{a_1}\frac{\partial}{\partial t}\right)^{\alpha} u_{\alpha}(x,t)=
					\left(\frac{1-a_1}{2} \right)^\alpha \frac{\partial^2}{\partial x^2}u_{\alpha}(x,t),
					\quad \alpha \in (0,1),\\
					u_{\alpha}(x,0)=\delta(x),
				\end{cases}
			\end{equation}
			is given by
			\begin{equation}
				\label{dueg}
				u_{\alpha}(x,t)=\frac{1}{2^{1-\alpha/2}t^{\frac{1-a_1}{2}\alpha}}
				W_{-\alpha/2, 1-\alpha/2}\left(-\frac{2^{\alpha/2}|x|}{t^{\frac{1-a_1}{2}\alpha}}\right).
			\end{equation}
		\end{remark}

		\begin{remark}
			If we consider the complete McBride operator of order one, that is $(t^{a_1}\frac{\partial}{\partial t} t^{a_2})$,
			with $a_1+a_2<1$, it is easy to prove that
			the function
			\begin{align}
				h_{\alpha}(x,t)=\frac{1}{2^{1-\alpha/2}t^{\frac{1-a_1-a_2}{2}\alpha+a_2}}
			    W_{-\alpha/2, 1-\alpha/2}\left(-\frac{2^{\alpha/2}|x|}{t^{\frac{1-a_1-a_2}{2}\alpha}}\right)
			\end{align}
			solves the non-homogeneous pde
			\begin{align}\label{bhoo}
				\left(t^{a_1}\frac{\partial}{\partial t}t^{a_2}\right)^{\alpha} h_{\alpha}(x,t)=
			    \frac{\partial^2}{\partial x^2}h_{\alpha}(x,t)+\frac{t^{-\alpha(1-a_1-a_2)-a_2}}{(1-a_1-a_2)^{-\alpha}}
			    \frac{\delta(x)}{\Gamma(1-\alpha)}.
			\end{align}
			
			Indeed, by taking the Fourier transform of \eqref{bhoo} we have that
		    \begin{equation}
			    \left(t^{a_1}\frac{\partial}{\partial t}t^{a_2}\right)^{\alpha}
			    \widehat{h}_{\alpha}(\beta,t)=
			    -\beta^2
			    \widehat{h}_{\alpha}(\beta,t)+\frac{t^{-\alpha(1-a_1-a_2)-a_2}}{(1-a_1-a_2)^{-\alpha}}
			    \frac{1}{\Gamma(1-\alpha)},
		    \end{equation}
		    whose solution is given by
		    \begin{equation}
		    	\label{ft}
			    \widehat{h}_{\alpha}(\beta,t)= \frac{1}{t^{a_2}}E_{\alpha,
			    1}\left(-\frac{\beta^2
			    t^{\alpha(1-a_1-a_2)}}{(1-a_1-a_2)^{\alpha}}\right).
		    \end{equation}
		    In view of the fact that
		    \begin{align}
			    \left(t^{a_1}\frac{\partial}{\partial
			    t}t^{a_2}\right)^{\alpha}t^\gamma&=
			    (1-a_1-a_2)^{\alpha}t^{-\alpha(1-a_1-a_2)}I^{\frac{a_2}{1-a_1-a_2},-\alpha}_{1-a_1-a_2}t^\gamma\\
			    \nonumber &=(1-a_1-a_2)^{\alpha}\frac{\Gamma\left(\frac{1-a_1+\gamma}
			    {1-a_1-a_2}\right)}{\Gamma\left(\frac{1-a_1+\gamma}{1-a_1-a_2}-\alpha\right)}t^{\gamma-\alpha(1-a_1-a_2)},
		    \end{align}
		    we have that
		    \begin{align}
			    \left(t^{a_1}\frac{\partial}{\partial
			    t} t^{a_2}\right)^{\alpha} & \frac{1}{t^{a_2}}E_{\alpha,
			    1}\left(-\frac{\beta^2
			    t^{\alpha(1-a_1-a_2)}}{(1-a_1-a_2)^{\alpha}}\right)\\
			    \nonumber &=(1-a_1-a_2)^{\alpha}
			    \sum_{k=0}^{\infty}\frac{(-\beta^2)^{k}}{(1-a_1-a_2)^{\alpha k}}
			    \frac{t^{\alpha(1-a_1-a_2)k-\alpha(1-a_1-a_2)-a_2}}{\Gamma\left(\alpha k+1-\alpha\right)} \\
			    &\nonumber=-\sum_{k=-1}^{\infty}\frac{(-1)^k(\beta)^{2k+2}}{(1-a_1-a_2)^{\alpha k}}
			    \frac{t^{\alpha(1-a_1-a_2)k-a_2}}{\Gamma\left(\alpha
			    k+1\right)}\\
			    \nonumber &=-\beta^2\widehat{u}_{\alpha}+
			    \frac{t^{-\alpha(1-a_1-a_2)-a_2}}{(1-a_1-a_2)^{-\alpha}}
			    \frac{1}{\Gamma(-\alpha+1)}.
		    \end{align}
		    By inverting the Fourier transform we obtain the claimed result.
			The case $a_1+a_2 = 1 $ can be treated by using the results discussed by Lamb and McBride in \cite{lam}.
			In this case, a logarithmic deterministic time-change leads to the solution.
			Finally, in the case $a_1+a_2 >1$ complex coefficients appear and therefore we neglect this case.
		\end{remark}
    	
    	If we consider the space-time fractional equation
    	\begin{equation}
    		\label{cis}
    			{}^C\left(t^{1-2H}\frac{\partial}{\partial t}\right)^{\alpha} u_{\alpha}(x,t)=
    			H^{\alpha}\frac{\partial^\nu}{\partial |x|^{\nu}}u_{\alpha}(x,t), \quad \nu \in (0,2), \ x \in \mathbb{R}
    		\end{equation}
    	where $\partial^{\nu}/\partial |x|^{\nu}$ denotes the Riesz fractional derivative
        \begin{align}
            \frac{\partial^\nu}{\partial |x|^\nu} g(x) = -\frac{1}{\Gamma(m-\nu)} \frac{1}{2\cos\frac{\pi\nu}{2}}\frac{\partial^m}{\partial x^m}
            \int_{-\infty}^{+\infty} \frac{1}{\left|x-s\right|^{1+\nu-m}} g(s) \mathrm ds, \qquad m-1<\nu<m, \ x>0
        \end{align}
        where $m$ is the ceiling of $\nu$, we have that its fundamental solution reads
        \begin{align}
			u_{\alpha}(x,t)&=\frac{1}{2\pi} \int_{\mathbb{R}}
			e^{-i\beta x}E_{\alpha,1}\left(-\frac{|\beta|^{\nu} t^{2H\alpha}}{2^{\alpha}}\right)\mathrm{d}\beta\\
	        \nonumber &=\frac{\sin \alpha\pi}{\pi} \int_0^\infty \frac{r^{\alpha-1}}{r^{2\alpha}+2r^\alpha\cos \alpha \pi+1}
	        \int_{\mathbb{R}} e^{-i\beta x} e^{-|\beta|^{\nu/\alpha}\frac{t^{2H}}{2}r}\mathrm d\beta
	        \, \mathrm dr.
        \end{align}
        The expression $u_{\alpha}(x,t)$ coincides with the distribution of the time-changed
        r.v.\ $Y_{\nu/\alpha}\left(\frac{t}{2}\mathcal{W}_\alpha (t)\right)$ where  
        $Y_{\nu/\alpha}$ is a symmetric stable r.v.\ of order $0<\nu/\alpha<2$ and 
        $\mathcal{W}_\alpha (t)$ is the Lamperti r.v. with parameter $\alpha$ (that is the ratio of two independent
        positively skewed stable r.v.'s of order $0<\alpha<1$), possessing density 
        \begin{equation}
	        P\{\mathcal{W}_{\alpha}\in dr\}/dr= \frac{\sin \pi \alpha}{\pi}\frac{r^{\alpha-1}}{r^{2\alpha}+1+2r^{\alpha}
            \cos \pi \alpha}, \quad r>0, \alpha \in (0,1).
        \end{equation}

		\subsection{Relation with the generalized grey Brownian motion (ggBm)}
		
			The generalized grey Brownian motion (ggBm) was recently introduced and studied by A.~Mura and coauthors in
			\cite{Mu1,Mu2,Mu3} as a family of non-Markovian stochastic processes for 
			anomalous fast or slow diffusions.
		 
			The marginal density function of the ggBm is given by 
			\begin{equation}
				\label{gb1}
				P(x,t)= \frac{1}{2t^{\gamma/2}}W_{-\frac{\delta}{2},1-\frac{\delta}{2}}
				\left(-\frac{|x|}{t^{\gamma/2}}\right), \qquad \gamma\in (0,2], \: \delta \in (0,1].
			\end{equation}  
			The ggBm includes for $\delta= 1$ and $\gamma \in (0,2]$ the fractional Brownian motion, for 
			$0<\delta=\gamma<1$ the grey Brownian motion introduced in the literature by Schneider \cite{Sch} and 
			for $\gamma= \delta =1$ the standard Brownian motion.
			We observe that the density \eqref{gb1} coincides with \eqref{due} for $\delta = \alpha$ and $\gamma/2= H\alpha$ (up to a constant due
			to the scale invariance of the Wright function). This means that 
			the equation \eqref{cas} governs a ggBm of parameters $(\alpha, H\alpha)$
			explicitely dependent on the Hurst parameter of the fractional Brownian motion. 
		
			In his PhD thesis \cite{Muth}, A.~Mura introduced the fractional equation governing the probability 
			density of the ggBm, that is given by
			\begin{equation}
				\label{gb2}
				P(x,t)= P(x,0)+ \frac{1}{\Gamma(\delta)}\frac{\gamma}{\delta}
				\int_0^t \tau^{\frac{\gamma}{\delta}-1}(t^{\gamma/\delta}-\tau^{\gamma/\delta})^{\delta-1}
				\frac{\partial^2}{\partial x^2}P(x,\tau)\,\mathrm{d}\tau.
			\end{equation} 	
			Equation \eqref{gb2} is an integro-differential equation involving the Erd\'elyi--Kober integral.
		
			We now show the equivalence between the fractional equation
			\begin{equation}
				{}^C\left(t^{1-2H}\frac{\partial}{\partial t}\right)^{\alpha} P(x,t)= H^{\alpha}
				\frac{\partial^2}{\partial x^2}P(x,t)
			\end{equation} 
			and the master equation \eqref{gb2} introduced by A.~Mura, in the case of $\delta = \alpha $ and $\gamma/2 = H\alpha$.
		
			From \eqref{pot1} and by using \eqref{cia}, we have that
			\begin{align}
				{}^C\left(t^{1-2H}\frac{\partial}{\partial t}\right)^{\alpha} P(x,t)
				& = \left(t^{1-2H}\frac{\partial}{\partial t}\right)^{\alpha}\left(P(x,t)-P(x,0)\right)\\
				\nonumber &=
				(2H)^{\alpha}t^{-2H\alpha}
				I^{0,-\alpha}_{2H}\left(P(x,t)-P(x,0)\right) \notag \\
				&=H^{\alpha}\frac{\partial^2}{\partial x^2}P(x,t), \notag
			\end{align}
			and therefore,
			\begin{equation}
				\label{gb3}
				I_{2H}^{0, -\alpha}\left(P(x,t)-P(x,0)\right)= \frac{t^{2H\alpha}}{2^\alpha}\frac{\partial^2}{\partial x^2}
				P(x,t).
			\end{equation} 
			We now recall the following property of the Erd\'elyi--Kober integral (see \cite{mca}, Theorem 2.7, page 523)
			\begin{equation}
				\left(I_m^{\eta, \alpha}\right)^{-1}= I_m^{\eta+\alpha, -\alpha}.
			\end{equation}
			Thus, in our case, we have that the inverse of the operator appearing in the left hand side of \eqref{gb3} is given by
			\begin{equation}
				\label{inver}
				\left(I_{2H}^{0, -\alpha}\right)^{-1}= I_{2H}^{-\alpha, \alpha}.
			\end{equation} 
			Finally, by applying the inverse operator \eqref{inver} to both sides of \eqref{gb3}, we arrive at
			\begin{align}
				P(x,t)-P(x,0)&= \frac{I_{2H}^{-\alpha, \alpha}}{2^{\alpha}}\left(t^{2H\alpha}\frac{\partial^2}{\partial x^2}P(x,t)\right)\\
				\nonumber &= \frac{2^{1-\alpha}H}{\Gamma(\alpha)}\int_0^t \tau^{2H-1}\left(t^{2H}
				-\tau^{2H}\right)^{\alpha-1}\frac{\partial^2}{\partial x^2}P(x,t)\,
				\mathrm{d}\tau,
			\end{align}
			which coincides with \eqref{gb2} for $\delta= \alpha$ and $\gamma = 2H\alpha$, up to a multiplicative constant.
		
			A further approach highlighting the relation between \eqref{gb2} and equations involving
			Erd\'elyi--Kober derivatives was recently discussed by Pagnini in \cite{gianni}. 
			To conclude, we remark that the general McBride theory for fractional powers of hyper-Bessel-type
			operators represents a convenient framework to derive the governing equation of the ggBm considered as a fractional
			generalization of the diffusion equation governing the fractional Brownian motion.

		\subsection{Relation with higher order heat equations with time-varying coefficients}
		
			\label{mala}
		    In this section we consider the relation between the solutions
		    of Theorem 3.1 and the fundamental solutions of higher order
		    heat-type equations.
		    
		    We start from a special case of \eqref{cas}, corresponding to
		    $\alpha = 2/3$. We know, from \cite{Ors} (Remark 4.1,
		    pages 231--232), that the fundamental solution to the fractional
		    equation involving Caputo derivatives
		    \begin{equation}
			    {}^C  D_{0^+}^{2/3} u=\lambda^2 \frac{\partial^2 u}{\partial
				x^2}, \qquad \lambda>0, \: t\ge 0,
		    \end{equation}
		    is related to the fundamental solution of the linearized Korteweg--DeVries equation, that is
		    the third order heat equation
		    \begin{equation}
			    \frac{\partial v}{\partial t}= -\lambda^3\frac{\partial^3 v}{\partial
			    x^3}, \qquad \lambda > 0, \ t\geq 0.
		    \end{equation}
		    By following the same idea, we establish here a relationship between fractional equations and higher-order
		    heat-type equations with time-varying coefficients.
		    
		    \begin{theorem}
		    	\label{ciccio}
			    The solution $u_{2/k}(x,t)$ of \eqref{cas}, for any $k\in \mathbb{N}: k>2$, coincides with the fundamental 
			    solution $v(x,t)$ of the higher order diffusion equation with
			    time-dependent coefficients
			    \begin{equation}
				    t^{1-2H}\frac{\partial v}{\partial t}= (-1)^k H\frac{\partial^k v}{\partial
				    x^k}, \qquad x>0,\: t\ge 0.
			    \end{equation}
		    \end{theorem}
		
		    \begin{proof}
			    We first recall that in the special case $\alpha = 2/k$ the
			    solution of \eqref{cas} is given for all $x\in \mathbb{R}$ by
			    \begin{equation}
			    	\label{dueter}
				    u_{2/k}(x,t)=\frac{1}{2^{1-\frac{1}{k}}t^{2H/k}}\sum_{n=0}^{\infty}\frac{(-1)^n}{n!\Gamma(-\frac{n}{k}+1-\frac{1}{k})}
				    \frac{2^{n/k}|x|^n}{t^{\frac{2}{k}Hn}}.
			    \end{equation}
			    By direct calculation we have, for $x>0$, that
			    \begin{align}
				    t^{1-2H}\frac{\partial  u_{2/k}}{\partial t}&=
				    2^{1/k}H \sum_{n=0}^{\infty}\frac{(-1)^n 2^{n/k}
				    x^n t^{-\frac{2}{k}Hn-\frac{2}{k}H-2H}}{n!\Gamma(-\frac{n}{k}-\frac{1}{k})}
			    \end{align}
			    and this coincides with $(-1)^k H$ times the $k$-th order derivative with
			    respect to $x$, as claimed.
		    \end{proof}
			
		    \begin{remark}
			    As a special case of interest of Theorem \ref{ciccio}, we observe that the solution $u_{2/3}(x,t)$ of \eqref{cas} coincides
			    for $x>0$ with the
		        fundamental solution $v(x,t)$ of the linearized Korteweg--DeVries equation with a
		        time-dependent coefficient
		        \begin{equation}\label{kdvv}
			        t^{1-2H}\frac{\partial v}{\partial t}=- H\frac{\partial^3 v}{\partial
			        x^3}.
		        \end{equation}
			    In particular, it is possible to prove that \eqref{dueter} can be represented
			    in terms of Airy functions, following the same arguments of
			    Theorem 4.1 in \cite{Ors}, since the fundamental solution to \eqref{kdvv} can be written as
			    \begin{equation}\label{kdvv2}
			    v(x,t)=\frac{1}{\sqrt[3]{3t^{2H}}}Ai\left(\frac{x}{\sqrt[3]{\frac{3}{2}t^{2H}}}\right).
			    \end{equation}
			    Observe that the transformations $t' = t^{2H}$ and
			    $x'= x\sqrt[3]{2}$, reduce equation \eqref{kdvv} to 
			    \begin{equation}\nonumber
			    \frac{\partial v}{\partial t'}= -\frac{\partial^3 v}{\partial x'^3},
			    \end{equation}
			    whose fundamental solution is 
			    \begin{equation}\nonumber
 				v(x',t')=\frac{1}{\sqrt[3]{3t'}}Ai\left(\frac{x'}{\sqrt[3]{3t'}}\right)
			    \end{equation}
			    and thus we have \eqref{kdvv2}.
		    \end{remark}

    \section{Diffusion equation with time-dependent coefficients involving Caputo derivatives}

      	In this section we adopt a different fractional generalization of the time-varying diffusion equation, involving
      	Caputo time-fractional derivatives. Our aim is then to compare analytical and probabilistic results relative
      	to these two different approaches.
      	 
   		In more detail, here we consider the time-fractional diffusion equation
		\begin{equation}
			\label{multifr}
			\begin{cases}
				{}^C D^\nu_{0^+} u \left( x,t \right) =
				H t^{2H-1} \frac{\partial^2}{\partial x^2} u \left( x,t \right),
				\qquad \nu \in \left( 0,1 \right], \: x \in \mathbb{R}, \: t >0,\\
				u \left( x,0 \right) = \delta \left( x \right),
			\end{cases}
		\end{equation}
		with $0 <H <1$.
		The derivative operator appearing in \eqref{multifr} is the Caputo fractional
		derivative.
		
		For the utility of the reader, we first recall the following result (see \cite{kilbas}, pag.232-233).
		\begin{lem}
			\label{result}
			The function 
			\begin{align}
				\label{mlgen}
				f(t)=E_{\nu, 1+\frac{\gamma}{\nu}, \frac{\gamma}{\nu}} \left(
				\lambda t^{\nu+\gamma} \right) = 1+ \sum_{k=1}^{\infty}\left(\lambda t^{\nu+\gamma}\right)^k
				\prod_{j=0}^{k-1}\frac{\Gamma(\nu j+\gamma j +\gamma+1)}{\Gamma(\nu j+\gamma j+\gamma+\nu+1)},
			\end{align}
			is a solution to the Cauchy problem
			\begin{equation}\label{marted}
				\begin{cases}
					{}^C D_{0^+}^\nu f(t)= \lambda t^{\gamma}f(t), & t\geq 0,\: \nu \in (0,1],\: \gamma>-\nu,\\
					f(0)=1.
				\end{cases}
			\end{equation}
		\end{lem}
	    \begin{proof}
		    We have that
            \begin{align}
	            {}^C D_{0^+}^\nu f(t) & =
	            \sum_{k=1}^\infty \lambda^k \:{}^C D_{0^+}^\nu \left[ t^{k \left( \nu +
	     		\gamma \right)} \right] \prod_{j=0}^{k-1}
	            \frac{\Gamma \left( \nu j+\gamma j +\gamma +1 \right)}{
	            \Gamma \left( \nu j + \gamma j+\gamma +\nu +1 \right)}
	            \\
	            & = \sum_{k=1}^\infty \lambda^k \frac{\Gamma \left(
	            k \nu + k \gamma +1 \right)}{\Gamma \left( k \nu
	            +k \gamma +1 -\nu \right)} t^{k \left( \nu + \gamma
	            \right) - \nu}
	            \prod_{j=0}^{k-1} \frac{
	            \Gamma \left( \nu j +j \gamma +\gamma +1 \right)}{
	            \Gamma \left( \nu j+ \gamma j+\gamma +\nu +1 \right)}
	            \nonumber \\
	            & = \lambda t^\gamma \sum_{r=0}^\infty \lambda^r
	            t^{r \left( \nu+\gamma \right)} \frac{\Gamma \left(
	            r \left( \nu +\gamma \right) +\nu +\gamma +1 \right)}{
	            \Gamma \left( r \left( \nu+\gamma \right)+\gamma +1 \right)}
	            \prod_{j=0}^r \frac{\Gamma
	            \left(\nu j+ j\gamma +\gamma+1
	            \right) }{\Gamma \left(\nu j+\gamma j +\gamma +\nu +1
	            \right)} \nonumber \\
	            & = \lambda t^\gamma \left[ 1+ \sum_{r=1}^\infty
	            \lambda^r t^{r \left( \nu+\gamma \right) }
	            \prod_{j=0}^{r-1} \frac{\Gamma \left(
	            \nu j+\gamma j +\gamma+1
	            \right)}{\Gamma \left( \nu j +\gamma j + \gamma +\nu +1
	            \right)} \right] \nonumber \\
	            &= \lambda t^{\gamma}f(t)\nonumber .
            \end{align}
		\end{proof}
		The function \eqref{mlgen} is the generalized Mittag--Leffler function first introduced in
		\cite{kilsa1,kilsa2,kilsa3}. An application of this result in fractional elasticity has been recently discussed in \cite{meccanica}.
		
		We are now ready to state the following
					                
		\begin{thm}
			\label{levecchio}
			A solution
			$u(x,t)$ to \eqref{multifr} can be written as
			\begin{align}
				\label{solutiona}
				u(x,t)
				= \frac{1}{2 \pi} \int_{-\infty}^{+\infty} e^{-i \beta x}
				E_{\nu, 1+\frac{2H-1}{\nu}, \frac{2H-1}{\nu}} \left(
				- H \beta^2 t^{\nu+2H-1} \right) \mathrm{d}\beta. 
			\end{align}
		\end{thm}			
		\begin{proof}			
			By taking the Fourier transform of \eqref{multifr} we obtain
			the differential equation
			\begin{equation}
				\label{fourierfr}
				\begin{cases}
					{}^C D_{0^+}^\nu U \left( \beta,t \right) =
					- H t^{2H-1} \beta^2 U \left( \beta, t \right),
					\qquad t>0, \: \nu \in (0,1], \\
					U (\beta,0)=1.
				\end{cases}
			\end{equation}
			Equation \eqref{fourierfr} is solved by (see Lemma \ref{result} and \cite{kilbas}, page 233, examples
			4.11 and 4.12)
			\begin{equation}
				\label{solutionfr1}
				U(\beta,t) = E_{\nu, 1+\frac{2H-1}{\nu}, \frac{2H-1}{\nu}} \left(
				- H \beta^2 t^{\nu+2H-1} \right).
			\end{equation}
			Moreover, the solution is proved to be unique for
			$1/2 \leq H <1$ (see \cite{kilbas}, page 233). 
			By taking the inverse Fourier transform we obtain the claimed result.
		\end{proof}
		
		\begin{rem}
			\label{rem1}
			An interesting case of \eqref{solutionfr1} is given by the choice $H= 1/4$, $\nu= 3/4$. In this case, from \eqref{solutionfr}
			we have that
			\begin{align}
				U(\beta,t) = E_{\frac{3}{4}, \frac{1}{3}, -\frac{2}{3}} \left(-\frac{1}{4}
				\beta^2 t^{1/4}\right) = 1+\sum_{k=1}^{\infty}\left(-\frac{\beta^2t^{1/4}}{2^2}\right)^k
				\prod_{j=0}^{k-1}\frac{\Gamma(\frac{j}{4}+\frac{1}{2})}{\Gamma(\frac{j}{4}+\frac{5}{4})}.
			\end{align}
			Since
			\begin{align}
				\prod_{j=0}^{k-1}\frac{\Gamma(\frac{j}{4}+\frac{1}{2})}{\Gamma(\frac{j}{4}+\frac{5}{4})}
				= \frac{\Gamma(1/2)\Gamma(3/4)}{\Gamma(\frac{k+2}{4})\Gamma(\frac{k+3}{4})\Gamma(\frac{k+4}{4})}
			\end{align}
			and using the well-known multiplication formula for the Gamma function
			\begin{equation}\nonumber
			\Gamma(z)\Gamma\left(z+\frac{1}{m}\right)\Gamma\left(z+\frac{2}{m}\right)\dots\Gamma\left(z+\frac{m-1}{m}\right)
			= (2\pi)^{(m-1)/2}m^{\frac{1}{2}-mz}\Gamma(mz),
			\end{equation}
			we have that 
			\begin{align}
				\prod_{j=0}^{k-1}\frac{\Gamma(\frac{j}{4}+\frac{1}{2})}{\Gamma(\frac{j}{4}+\frac{5}{4})}
				= \frac{\Gamma(\frac{k+1}{4})\Gamma(3/4)}{\pi 2^{-2k+\frac{1}{2}} k!}.
			\end{align}
			Therefore,
			\begin{align}
				\label{lun1}
				U(\beta,t) &= E_{\frac{3}{4}, \frac{1}{3}, -\frac{2}{3}} \left(-\frac{1}{4}
				\beta^2 t^{1/4}\right) = 
				1+\frac{\Gamma(3/4)}{\pi \sqrt{2}}\sum_{k=1}^{\infty}\left(-\beta^2 t^{1/4}\right)^k
				\frac{\Gamma\left(\frac{k+1}{4}\right)}{k!}\\
				\nonumber &= 1+\frac{\Gamma(3/4)}{\pi \sqrt{2}}\int_0^{\infty}e^{-w}w^{-3/4}
				\sum_{k=1}^{\infty}\frac{\left(-\beta^2 w^{1/4}t^{1/4}\right)^k}{k!}\mathrm{d}w\\
				\nonumber & =\frac{\Gamma(3/4)}{\pi \sqrt{2}}\int_0^{\infty}e^{-w}w^{-3/4} e^{-\beta^2\sqrt[4]{tw}} \mathrm{d}w.				
			\end{align}
			Finally, we have
			\begin{align}
				u(x,t)&= \frac{\Gamma(3/4)}{\pi \sqrt{2}}
				\int_0^{\infty}e^{-w}w^{-3/4}\frac{e^{-\frac{x^2}{4\sqrt[4]{tw}}}}{\sqrt{4\pi\sqrt[4]{tw}}}\mathrm{d}w\\
				\nonumber &=\frac{\sqrt{2}\Gamma(3/4)}{\pi\sqrt[4]{t}}\int_{0}^{\infty}\frac{e^{-\frac{x^2}{2w}}}{\sqrt{2\pi w}}
				\dot e^{-\frac{w^4}{2^4t}}dw,
			\end{align}
			that is the probability law of the r.v. $B(W_t)$, 
			where 
			\begin{equation}
				\nonumber
				\mathbb{P}\{W_t\in \mathrm{d}z\} / \mathrm dz= \frac{\sqrt{2}\Gamma(3/4)e^{-\frac{z^4}{2^4 t}}}{\pi\sqrt[4]{t}}, \qquad z>0,
			\end{equation}
			and $B$ is a standard Brownian motion.
		\end{rem}
				
		\begin{rem}
			For $\nu = 1$ we show that the solution $u(x,t)$ in equation \eqref{solutiona} 
			coincides with the density function of the fractional Brownian motion for a fixed time $t$. 
			Indeed, for $\nu= 1$, formula \eqref{solutionfr1} yields
			\begin{equation}
				U(\beta,t) = \exp\bigg\{-\frac{\beta^2 t^{2H}}{2}\bigg\},
			\end{equation}
			which is the characteristic function of  $B_H(t)$, as can be checked from \eqref{mlgen} for $\lambda =-H\beta^2$
			and $\gamma= 2H-1$ (or directly from \eqref{marted}).
		\end{rem}
		
		\begin{rem}
			For $1+\frac{2H-1}{\nu}= 0$, the characteristic function \eqref{solutionfr1} becomes
			\begin{equation}
				\label{m0}
				U(\beta, t)= \frac{1}{1+H\beta^2\Gamma(1-\nu)}, \qquad H\beta^2\Gamma(1-\nu)<1, \; H\in \left(0,\frac{1}{2}\right),
			\end{equation}
			and is independent from $t$. Indeed, in this case we have that
			\begin{align}
				\nonumber
				U(\beta,t)& = E_{\nu, 0, -1} \left(
				- H \beta^2 \right)=1+\sum_{k=1}^\infty \left(-H\beta^2\right)^k \prod_{j=0}^{k-1}
				\frac{ \Gamma \left(1-\nu \right)}{\Gamma \left(1\right)}\\
				\nonumber &	= \sum_{k=0}^\infty \left(-H\beta^2\Gamma(1-\nu)\right)^k=
				\frac{1}{1+H\beta^2\Gamma(1-\nu)}, 
			\end{align}
			when $H\beta^2\Gamma(1-\nu)<1$.
			
			We observe that, even if the function \eqref{m0} is independent from $t$, it satisfies the time-fractional differential equation
			\begin{equation}
				t^{\nu}D^\nu_{0^+} f= \frac{f}{\Gamma(1-\nu)}
			\end{equation}	
			involving the Riemann--Liouville derivative. We recall that the Riemann--Liouville derivative of a constant function 
			does not vanish and in more detail
			\begin{equation}
				D^\nu_{0^+} \text{const.}= \frac{\text{const.} \ t^{-\nu}}{\Gamma(1-\nu)}.
			\end{equation}
		\end{rem}
	
		\begin{rem}
			For $\nu = 2H-1= \frac{1}{2}$ and therefore $H= \frac{3}{4}$, the characteristic function \eqref{solutionfr1} becomes
			\begin{equation}
				\label{binoma}
				U(\beta, t)= 1+\sum_{k=1}^{\infty}(-H\beta^2t)^k\prod_{j=0}^{k-1}\frac{\sqrt{\pi}}{2^{2j+1}}\binom{2j+1}{j+1}
				= 1+\sum_{k=1}^{\infty}\left(-H\sqrt{\pi}\beta^2 t\right)^k\prod_{j=1}^{k}\frac{1}{2^j}\binom{2j}{j}.
			\end{equation}
			We observe that the coefficients in \eqref{binoma} have the following probabilistic interpretation
			\begin{equation}
				\label{domen}
				\prod_{j=1}^{k}\frac{1}{2^j}\binom{2j}{j}
				= \mathbb{P}\bigg\{\bigcap_{j=1}^k\left[\mathfrak{B}\left(2j, \frac{1}{2}\right)=j\right]\bigg\}.
			\end{equation}
			In \eqref{domen} by $\mathfrak{B}^j$ we denote independent binomial r.v.'s with parameters $2j$ and 
			$1/2$.
			For large values of $j$, we know that $\mathbb{P}\{\mathfrak{B}\left(2j, \frac{1}{2}\right)=j\}= 1/\sqrt{\pi j}$
			and thus the coefficients of \eqref{binoma} vanish as $1/\sqrt{k!}$.
		\end{rem}
		
		\subsection{Higher order extension}
			Here we consider the higher-order time-fractional equation
		 	    \begin{align}
		        	\label{gedit}
		          	\begin{cases}
		          		{}^C D_{0^+}^\nu f(x,t) = c_k2Ht^{2H-1}
		          	    \frac{\partial^k}{\partial x^k} f(x,t), &
		          	    \nu \in (0,1], \: 0 <H \leq 1,\\
		          	    f(x,0)=\delta(x), \\
		          	\end{cases}
		        \end{align}
		        and $k$ is an integer greater
		        than $2$.
		        The coefficient $c_k=(-1)^{1+k/2}$ if $k$ is even and $c_k=\pm 1$ otherwise.
		        The fractional derivative appearing in \eqref{gedit} is the Caputo fractional
		        derivative.
		        
		        \begin{thm}
					\label{help}
		          	The solution $f(x,t)$, $x \in \mathbb{R}$, $t\ge 0$, to \eqref{gedit} is
		          	\begin{align}
		          		\label{solution}
		          	    	& f(x,t) = \frac{1}{2 \pi} \int_{\mathbb{R}} e^{-i \beta x}
		          	        E_{\nu, 1+\frac{2H-1}{\nu}, \frac{2H-1}{\nu}} \left(c_k 2H (-i\beta)^k t^{\nu+2H-1} \right) \mathrm d\beta.
		          	\end{align}
		        \end{thm}
		        \begin{proof}
		        	Applying the Fourier transform
		          	\begin{align}
		          		F(\beta, t) = \int_{\mathbb{R}} e^{i\beta x} f(x,t) \, \mathrm dx,
		          	\end{align}
		          	to \eqref{gedit}, we have
		          	\begin{align}
		          		\label{geditf}
		          	    \begin{cases}
		          	    	{}^C D_{0^+}^\nu F \left( \beta,t \right) =
		          	        c_k2H t^{2H-1} (-i\beta)^k F \left( \beta, t \right), \\
		          	        F (\beta,0)=1.
		          	    \end{cases}
		          	\end{align}
		          	The solution to \eqref{geditf} is immediately obtained:
		          	\begin{equation}
		          		\label{solutionfr}
		          	    F(\beta,t) =  E_{\nu, 1+\frac{2H-1}{\nu}, \frac{2H-1}{\nu}} \left(c_k 2H (-i\beta)^k t^{\nu+2H-1} \right) ,
		          	\end{equation}
		          	Recall that the solution is proved unique for
		          	$1/2 \leq H \leq 1$. From \eqref{solutionfr}
		          	it is immediate to obtain \eqref{solution}.
				\end{proof}
		
				In the special case $H= 1/4$, $\nu= 3/4$ and $k = 2n$,
				the Fourier transform of the solution of \eqref{gedit}
				reads 
				\begin{align}
					F(\beta,t)= \frac{\Gamma(\frac{3}{4})\sqrt{2}}{\pi\sqrt[4]{t}}
					\int_0^{\infty}e^{-\beta^{2n} y}e^{-\frac{y^4}{2^4t}}dy.
				\end{align}
				Thus the inverse Fourier transform coincides with the law 
				of a pseudoprocess of order $2n$ evaluated at the random time $W_t$.
				In a similar way the odd-order case can be treated. 
				For a wide discussion about pseudoprocesses we refer for example to \cite{Lachal} and the references therein.

		\subsection{Space-time bifractional equations}
	
	        Theorem \ref{levecchio} can be modified by introducing the Riesz fractional derivative,
	        in the right hand side of \eqref{multifr}, thus obtaining
	        \begin{align}
	            \label{gredriesz}
	            \begin{cases}
	                {}^C D_{0^+}^\nu f(x,t) = Ht^{2H-1}
	                \frac{\partial^\alpha}{\partial |x|^\alpha} f(x,t), &
	                \nu \in (0,1], \: 0<\alpha<2, \: 0<H< 1,\\
	                f(x,0)=\delta(x).
	            \end{cases}
	        \end{align}
	        This equation has been object of recent interest within the framework of anomalous diffusion processes \cite{Bologna1,Bologna2}.
	        Similarly to Theorem \ref{help}, we apply the
	        Fourier transform to \eqref{multifr}, arriving at
	        \begin{align}
	            \label{fouriesz}
	            \begin{cases}
	                {}^C D^\nu_{0^+} F(\beta,t) = -Ht^{2H-1}
	                | \beta |^\alpha F(\beta, t), &
	                \nu \in (0,1], \: 0<\alpha<2, \\
	                F(\beta,0)=1.
	            \end{cases}
	        \end{align}
	        The solution to \eqref{fouriesz} is thus
	        \begin{align}
	            \label{soluriesz}
	            & f(x,t) = \frac{1}{2 \pi} \int_{\mathbb{R}} e^{-i \beta x}
	            E_{\nu, 1+\frac{2H-1}{\nu}, \frac{2H-1}{\nu}} \left(-H |\beta|^{\alpha} t^{\nu+2H-1} \right) \mathrm d\beta.
	        \end{align}
			We observe that \eqref{soluriesz} coincides with equation (63) of \cite{Bologna2} with a suitably 
			adaptation of the notation.
			\begin{rem}
				When $H= 1/4$, $\nu= 3/4$, on the basis of the results discussed in Remark \ref{rem1} (see equation \eqref{lun1}), we have that
				\begin{equation}
					F(\beta,t) = \frac{2\sqrt{2}\Gamma(3/4)}{\pi \sqrt[4]{t}}
					\int_0^{\infty} e^{-z^4/t}e^{-|\beta|^{\alpha}z}\mathrm dz,
				\end{equation}
				that coincides with the probability law of the r.v.
				$Y_{\alpha}(\mathfrak{W}_t)$, where 
				\begin{equation}
					\mathbb{P}\bigg\{\mathfrak{W}_t\in dz\bigg\}/\mathrm dz
					= \frac{2\sqrt{2}\Gamma(3/4)e^{-\frac{z^4}{ t}}}{\pi\sqrt[4]{t}}, \qquad z>0
				\end{equation}	
				and $Y_{\alpha}$ is a symmetric stable r.v.\ of order $\alpha$.
			\end{rem}			
					
	        \begin{rem}
	            When $H=1/2$, the solution \eqref{soluriesz}
	            reduces to
	            \begin{align}
	                \label{solred}
	                f(x,t) = \frac{1}{2 \pi} \int_{\mathbb{R}} e^{-i \beta x}
	                E_{\nu, 1} \left(
	                -\frac{1}{2} |\beta|^\alpha t^\nu \right) \mathrm d\beta.
	            \end{align}
	            The expansion of the Mittag--Leffler function yields
	            \begin{align}
	                f(x,t) & = \frac{1}{2\pi} \int_{\mathbb{R}} e^{-i\beta x} \frac{\sin \nu\pi}{\pi}
	                \int_0^\infty \frac{r^{\nu-1}}{r^{2\nu}+2r^\nu \cos \nu \pi +1}
	                e^{-\left(\frac{1}{2}\right)^{1/\nu}|\beta|^{\alpha/\nu}rt} \mathrm dr \, \mathrm d\beta \\
	                & = \frac{\sin \nu \pi}{2\pi^2} \int_0^\infty \frac{r^{\nu-1}}{r^{2\nu}+2r^\nu\cos \nu \pi+1}
	                \int_{\mathbb{R}} e^{-i\beta x} e^{-\left(\frac{1}{2}\right)^{1/\nu}|\beta|^{\alpha/\nu}rt}\mathrm d\beta
	                \, \mathrm dr. \notag
	            \end{align}
	            Then $f(x,t)$ is the distribution of the time-changed r.v.\ $Y_{\alpha/\nu}\left(\frac{t}{2^{\nu}}\mathcal{W}_\nu (t)\right)$
	            where $\mathcal{W}_\nu (t)$ is the Lamperti r.v.\ and 
	            $Y_{\alpha/\nu}$ is a symmetric stable r.v.\ of order $0<\alpha/\nu<2$. 
	        \end{rem}
	
	        \begin{rem}
	            If $\nu=\gamma \eta$, we can expand the Mittag--Leffler function appearing in \eqref{solred}
	            with an alternative formula (see \cite{polram} page 291), i.e.
	            \begin{equation}
	                \label{gen-mitta}
	                E_{\gamma \eta, 1} (-\vartheta t^{\gamma \eta}) =
	                \frac{\sin \gamma \pi}{\pi} \int_0^\infty
	                \frac{r^{\gamma-1}}{r^{2\gamma}+2r^\gamma \cos \gamma \pi+1}
	                E_{\eta,1} \left(-r \vartheta^{1/\gamma} t^\eta\right) \mathrm dr,
	            \end{equation}
	            where $\gamma \in (0,1]$, $\eta \in (0,1]$. In our case, for $H= 1/2$ we obtain that
	            \begin{align}
	           		&f(x,t) = \frac{1}{2\pi} \int_{\mathbb{R}} e^{-i\beta} \frac{\sin \gamma\pi}{\pi}
	            	\int_0^\infty \frac{r^{\gamma-1}}{r^{2\gamma}+2r^\gamma \cos \gamma \pi +1}
	            	E_{\eta,1} \left(-H^{1/\gamma}|\beta|^{\alpha/\gamma}rt^\eta\right)
	            	\mathrm dr \, \mathrm d\beta \\
	            	\nonumber & = \frac{\sin \gamma\pi}{2\pi^2} \int_0^\infty \frac{r^{\gamma-1}}{r^{2\gamma}
	            	+2r^\gamma\cos \gamma \pi+1}\int_{\mathbb{R}} e^{-i\beta x} E_{\eta,1}
	            	\left(-H^{1/\gamma}|\beta|^{\alpha/\gamma}rt^\eta\right)\mathrm d\beta
	            	\, \mathrm dr\\
	            	\nonumber &= \frac{\sin \gamma\pi \ \sin \eta \pi}{2\pi^3} \int_0^\infty \mathrm{d}r \frac{r^{\gamma-1}}{r^{2\gamma}
	            	+2r^\gamma\cos \gamma \pi+1}\int_{-\infty}^{+\infty} e^{-i\beta x} \mathrm{d}\beta\int_{0}^{\infty}
	            	\frac{y^{\eta-1}e^{-|\beta|^{\frac{\alpha}{\gamma \eta}}ty r^{1/\eta}}}{y^{2\eta}+1+2y^{\eta}\cos\pi \eta} \mathrm{d}y.
	            \end{align}
	            Therefore, for $H = 1/2$ and $0<\alpha/\gamma\eta<2$, $f(x,t)$ is the distribution of the time-changed random variable  
	            $Y_{\alpha/\gamma \eta}\left(t\mathcal{W}_\eta \mathcal{W}_\gamma^{1/\eta}\right)$,
	            where $\mathcal{W}(t)$ is the Lamperti r.v.\ and $Y_{\alpha/\gamma \eta}$ is 
	        	a symmetric stable law of order $\frac{\alpha}{\gamma \eta}$.
	        \end{rem}
	        
	        \begin{rem}
		        We observe that for $\nu = 1$, $\alpha \in (0,2)$ and
		        $H\in (0,1)$, from \eqref{soluriesz}
		        we obtain the law of a symmetric stable process with 
		        a deterministic time-change, whose characteristic function is given by
		        \begin{equation}
			        F(\beta,t)= \exp\bigg\{-\frac{t^{2H}}{2}|\beta|^{\alpha}\bigg\}.
		        \end{equation}
	        \end{rem}


\begin{thebibliography}{99}

		\bibitem{Allouba} H. Allouba, W. Zheng, Brownian-time processes: The PDE connection and the half-derivative
			generator. \emph{Ann. Probab.}, \textbf{29}:1780--1795, (2001)
			
		\bibitem{Bologna1} M. Bologna, A. Svenkeson, B.J. West, P. Grigolini, Diffusion in heterogeneous media:
			An iterative scheme for finding approximate solutions to fractional differential equations with time-dependent coefficients.
			\emph{Journal of Computational Physics}, \textbf{293}: 297--311, (2015)
		
		\bibitem{Bologna2} M. Bologna, B.J. West, P. Grigolini, Renewal and memory origin of anomalous diffusion:
			a discussion of their joint action. \emph{Phys. Rev. E}, \textbf{88}:062106, (2013)

		\bibitem{meccanica} E. Capelas de Oliveira,
        	F. Mainardi, J. Vaz Jr., Fractional models of anomalous relaxation based
         	on the Kilbas and Saigo function. \emph{Meccanica}, \textbf{49}(9):2049--2060, (2014)
         	
        \bibitem{meer} Z. Q. Chen, M. M. Meerschaert, E. Nane, Space-time fractional diffusion on bounded domains.
	         \emph{Journal of Mathematical Analysis and Applications}, \textbf{393}(2):479--488, (2012)
	         
	    \bibitem{diet} K. Diethelm, \emph{The Analysis of Fractional Differential Equations}, Springer, New
	    York, (2010).

		\bibitem{Dim66}
    		I. Dimovski, On an operational calculus for a differential operator.
		    \emph{Compt. Rendues de l'Acad. Bulg. des Sci.}, \textbf{21}(6):513--516, (1966)
		    
		\bibitem{mirko1} M. D'Ovidio, E. Orsingher,
			Compositions of processes and related partial differential equations. 
			\emph{Journal of Theoretical Probability}, \textbf{24}:342--375, (2011)

        \bibitem{mirko} M. D'Ovidio, Wright functions governed by fractional directional derivatives and fractional
            advection diffusion equations. \emph{Methods and Applications
            of Analysis}, \textbf{22}(1): 1--36, (2015)   
             
	    \bibitem{Fujita} Y. Fujita, Integrodifferential equation which interpolates the heat equation and the wave
			equation. II. \emph{Osaka J. Math.}, \textbf{27}:797--804, (1990)
     
		\bibitem{RG} R. Garra, E. Orsingher, F. Polito, Fractional Klein--Gordon Equations and Related Stochastic Processes.
			\emph{Journal of Statistical Physics}, \textbf{155}(4):777--809, (2014)
     
		\bibitem{Rogosin} R. Gorenflo, A.A. Kilbas, F. Mainardi, S.V. Rogosin, \emph{Mittag--Leffler functions.
     		Related Topics and Applications}. Springer Monographs in Mathematics, Berlin,
     		(2014)
     	
     	\bibitem{Turner} H. Jiang, F. Liub, I. Turner, K. Burrage, Analytical solutions for the multi-term time-space Caputo--Riesz
     		fractional advection-diffusion equations on a finite domain. \emph{Journal of Mathematical Analysis and Applications},
	     	389(2):1117--1127, (2012)
   
	    \bibitem{kilsa1} A.A. Kilbas, M. Saigo,
    		Fractional integrals and derivatives of Mittag--Leffler type functions.
    		\emph{Doklady Akademii Nauk Belarusi} \textbf{39}(4):22--26, (1995).
    
	   	\bibitem{kilsa2} A.A. Kilbas, M. Saigo,
	    	Solution of Abel integral equations of the second kind and differential equations of fractional order.
	    	\emph{Doklady Akademii Nauk Belarusi}, \textbf{39}(5):29--34, (1995).
    
	    \bibitem{kilsa3} A.A. Kilbas, M. Saigo,
	    	On solution of integral equation of Abel--Volterra type.
    		\emph{Differential and Integral Equations}, \textbf{8}(5):993--1011, (1995).
    		
   		\bibitem{kilbas}
			A.A. Kilbas, H.M. Srivastava, J.J. Trujillo,
			\emph{{Theory and Applications of Fractional Differential
			Equations}}. {Elsevier Science}, (2006)
    	
    	\bibitem{Lachal} A. Lachal, Distributions of Sojourn Time, Maximum and Minimum for Pseudo-Processes Governed by
    		Higher-Order Heat-Type Equations. \emph{Electronic Journal of Probability}, \textbf{20}(8):1--53, (2003) 
	            
        \bibitem{lam} W. Lamb, A.C.
	        McBride, On relating two approaches to fractional calculus.
	        \emph{Journal of Mathematical Analysis and Applications},
	        \textbf{132}:590--610, (1988)
	     
	    \bibitem{Luc} Y. Luchko, Initial-boundary-value problems for the generalized multi-term time-fractional diffusion equation.
		     \emph{Journal of Mathematical Analysis and Applications}, \textbf{374}(2), 538--548, (2011)
	             
        \bibitem{mainardi} F. Mainardi, Y. Luchko, G. Pagnini,
	        The fundamental solution of the space-time fractional diffusion equation.
	        \emph{Fractional Calculus and Applied Analysis}, \textbf{4}(2):153--192, (2001)
        
        \bibitem{Martinez} C. Martinez Carracedo, M. Sanz Alix,
       		\emph{The theory of fractional powers of operators}. North-Holland Mathematics Studies,
       		187. North-Holland Publishing Co., Amsterdam, (2001)  

        \bibitem{mc2} A.C. McBride, A theory of fractional integration for generalized functions.
            \emph{SIAM Journal on Mathematical Analysis}, \textbf{6}(3):583--599, (1975)

        \bibitem{mc1} A.C. McBride, \emph{Fractional calculus and integral transforms
            of generalised functions}. Pitman, London, (1979)
            
        \bibitem{mca} A.C. McBride, Fractional Powers of a Class of Ordinary Differential Operators.
			\emph{Proceedings of the London Mathematical Society}, \textbf{3}(45):519--546, (1982)   
		
		\bibitem{Muth} A. Mura, \emph{Non-Markovian Stochastic Processes and Their Applications: From Anomalous Diffusion
			to Time Series Analysis}, PhD Thesis, University of Bologna, (2008)

		\bibitem{Mu1} A. Mura, F. Mainardi, A class of self-similar stochastic processes with stationary increments to model anomalous diffusion
			in physics, \emph{Integral Transform and Special functions}, 20(3), 185--198, (2009)
		
		\bibitem{Mu2} A. Mura, G. Pagnini, Characterizations and simulations of a class of 
			stochastic processes to model anomalous diffusions, \emph{J. Phys. A: Math. Theor.}, 41(28), 285003, (2008)
		
		\bibitem{Mu3} A. Mura, M.S. Taqqu, F. Mainardi, Non-Markovian diffusion equations and processes: Analysis and
			simulations, \emph{Physica A}, 387(21):5033--5064, (2008)

        \bibitem{Ors} E. Orsingher, L. Beghin, Fractional diffusion equations and processes with randomly varying time.
        	\emph{Ann. Probab.}, \textbf{37}(1):206--249, (2009)
        	
        \bibitem{fepr} E. Orsingher, F. Polito, Some results on time-varying fractional partial differential equations and birth-death
	        processes. \emph{Proceedings of the XIII international conference on eventological mathematics and related fields},
	        Krasnoyarsk, 23--27, (2009)
		
		\bibitem{polram}
			E. Orsingher and F. Polito, Randomly Stopped Nonlinear Fractional Birth Processes. \emph{Stochastic Analysis and Applications}, 
			\textbf{31}(2):262--292, (2013)	
			
		\bibitem{gianni} G. Pagnini, Erd\'elyi-Kober 
			fractional diffusion, \emph{Fract. Calc. Anal. Appl.},
			15(1), 117--127, (2012)

        \bibitem{e-p} S.G. Samko, A.A. Kilbas, O.I. Marichev,
            \emph{Fractional integrals and derivatives}. Gordon and Breach Science, Yverdon, (1993)
            
        \bibitem{Sch} W.R. Schneider, Grey noise, in \emph{Stochastic Processes, Physics and Geometry}, World Scientific, 
	        Teaneck, 676--681, (1990)

		\bibitem{Wyss} W. Wyss, The fractional diffusion equation. \emph{J. Math. Phys.} \textbf{27}:2782--2785, (1986)

    \end{thebibliography}
\end{document}